\documentclass[11pt]{amsart}

\usepackage[top = 1in, left = 1.25in, right = 1.5in, bottom = 1in]{geometry}
\usepackage{enumitem}

\usepackage{amsmath, amsthm, amssymb, amsfonts}
\usepackage{amsrefs}
\usepackage{mathrsfs}
\usepackage{verbatim}
\usepackage{graphicx}
\usepackage{tikz}

\def\be{\begin{equation}}
\def\en{\end{equation}}

\def\max{\text{max}}

\definecolor{darkgreen}{rgb}{.1,.6,0}

\newcommand{\N}{\mathbb{N}}

\newtheorem{theorem}{Theorem}[section] 
\newtheorem{lemma}[theorem]{Lemma}     

\newtheorem{proposition}[theorem]{Proposition}

\theoremstyle{definition}
\newtheorem{definition}[theorem]{Definition}
\newtheorem{example}[theorem]{Example}
\newtheorem*{ack*}{Acknowledgment}

\theoremstyle{remark}
\newtheorem{remark}[theorem]{Remark}

\numberwithin{equation}{section}

\title{Arrival of information at a target set in a network}

\author{Karl Petersen}
\address{CB 3250 Phillips Hall,	University of North Carolina, Chapel Hill, NC 27599 USA}
\email{petersen@email.unc.edu}

\author{Ibrahim Salama}
\address{School of Business, North Carolina Central University, Durham, NC}
\email{isalama@nccu.edu}

\date{\today}

\begin{document}


\begin{abstract}
	We consider labelings of a finite regular tree by a finite alphabet subject to restrictions specified by a nonnegative transition matrix, propose an algorithm for determining whether the set of possible configurations on the last row of the tree is independent of the symbol at the root, and prove that the algorithm succeeds in a bounded number of steps, provided that the dimension of the tree is greater than or equal to the maximum row sum of the transition matrix. (The question was motivated by calculation of topological pressure on trees and is an extension of the idea of primitivity for nonnegative matrices.)
	\end{abstract}
	
\maketitle

{\tiny }
		\section{introduction}\label{sec:intro}
		Information being transmitted through a network or stored on it is of course subject to the network connections, and it may be further constrained by relations within the alphabet of symbols.
		Let $D=\{1,\dots,d\}$ be a finite alphabet, from which messages, which we take to be strings in $D^*=\cup \{D^n:n=0,1,\dots\}$, can be formed.
		Some simple restrictions on messages are {\em nearest neighbor constraints} determined by a nonnegative matrix $A$: symbols $a,b \in D$ are allowed to occupy the initial and terminal vertices of a directed edge of the network if and only if $A(a,b)>0$. 
		Such constraints may arise from physical properties of the network or the writing or reading mechanisms. 
		
		We are interested in which messages can be seen on (or ``arrive at") a specified target set of vertices $\mathcal T$ if a starting symbol $i \in D$ is assigned to an initial vertex $v_0$. 
		Define $\mathcal A(i, v_0,\mathcal T)$ to be the {\em arrival set at $\mathcal T$ given $i$ at the start, $v_0$}; namely, $\mathcal A(i,v_0,\mathcal T)$ is  the set of labelings of $\mathcal T$ that extend to valid (according to the restrictions imposed by the transition matrix $A$) labelings of the entire network that have label $i$ at the root.
		Can all initial symbols produce the same set of messages on the target set, that is, does $\mathcal A(i, v_0,\mathcal T)=\mathcal A(j, v_0,\mathcal T)$ for all $i,j \in D$? 
		In such a case we can say that the matrix $A$ is {\em $(v_0,\mathcal T$)-fair}. 
		
				Even more, is every possible set of configurations $D^{\mathcal T}$ achievable from every initial symbol {at $v_0$}? 
		If so, we can call $A$ {{\em $(v_0,\mathcal T)$-complete}}. 
		When the tree is $1$-dimensional, and so naturally identified with $\N$, the transition matrix $A$ is $v$-complete for vertices $v$ far enough from $v_0$ if and only if it is {\em primitive}: there is $p \in \N$ such that $A^p>0$. 
			 	(Wielandt \cite{Wielandt1950} determined the smallest such $p$, called the {\em exponent.} See also \cites{Dulmage1962,Schneider2002,Brualdi1991}.)
			 	
			 		In previous work \cite{PS2023}, it was discovered that $\mathcal T$-fairness of $A$ is a sufficient condition for existence as a limit of the topological pressure defined by $A$ on a tree. 
			 		{Thus this concept is of interest for the study of symbolic dynamics and statistical physics on trees, and potentially on other networks.}
			 	{A fundamental question, extending the idea of primitivity for matrices, asks: how can one determine whether a given transition matrix is or is not fair with respect to a given network and initial site?
			 		Here we seek conditions on $A$ that guarantee that $A$ is $\mathcal T$-fair when $v_0=\epsilon$ is the root of the tree and $\mathcal T=L_n$, the $n$'th row of the tree.	
			 	While we do not answer the question completely, we provide an algorithmic answer for regular trees under moderate assumptions---a first step, and an invitation to further investigation.} 
			 	
			 	
			 	The following section establishes necessary terminology and notation and begins the study of relations among symbols of the alphabet placed at the root. 
			 	We say that a symbol $i$ can be {\em replaced} by a symbol $j$ on a {tree} of height $n$ if for every labeling of the tree that has $i$ at the root there is a labeling which has $j$ at the root such that the two labelings agree on row $n$, that is, 
			 	$\mathcal A(i, \epsilon, L_n) \subseteq \mathcal A(j,\epsilon,L_n)$ (Definition \ref{def:relation}).
			 		 	We develop machinery for finding such relations by examining the transition matrix and in
			 	Section \ref{sec:procedure} formalize the process as a definite algorithm.  
			 	Section \ref{sec:works} includes our main result, Theorem \ref{thm:disc}: If the dimension of the tree is greater than or equal to the maximum row sum of the matrix, then the algorithm successfully finds all relations that exist, thereby answering the question of whether or not the matrix is ``fair" in the sense that the efficacy of transmission of information from the root of a tree to its last row is independent of the choice of initial symbol.
			 	Example \ref{ex:smallk} shows that the hypothesis involving dimension and row sum is necessary.

			 			 {\section{Basic definitions; the sets $P(k,n)$.}\label{sec:defs}}
			 				Recall that the {regular} $k$-tree $\tau$ ($k \geq 1$) is (or corresponds to) the set $K^*$ of $1$-sided infinite strings on $K=\{1, \dots, k\}$.
			 				A string of length $n$ corresponds to a site (or node or vertex) at level $n$ in the tree. 
			 				We denote by $L_n$ the set of sites at level $n$ and by $\Delta_n$ the set of all sites of length no more than $n$. 
			 				The empty string $\epsilon$ corresponds to the root, at level $0$.
			 		{For a site $u \in K^*$} and symbol $g \in K$, {we regard $(u,ug)$} as the {\em directed edge} with {\em source vertex} $u$ and {\em target vertex} $ug$. 
			 			Each vertex $ug, g \in K$, is a {\em successor} or {\em child} of the vertex $u$, and $u$ is the {\em predecessor} or {\em parent} of each $ug, g \in K$.
			 				
			 				A {\em labeling} of $\tau$ is a function $\lambda: \tau \to D$, for the finite alphabet $D=\{1, \dots , d\}$, $d \geq 1$. 
			 				We are considering labelings $\lambda$ that are consistent with a $d \times d$ nonnegative matrix $A$:			 			
			 				our set of allowed labelings is
			 				\be 
			 					X_A= \{\lambda \in  D^\tau: A(\lambda(x),\lambda(xg)) > 0 \quad\text{for all } x \in \tau, \, g=1, \dots,k \}.
			 				\en 
			 				Since if necessary $A$ can be replaced by a $0,1$ matrix that allows exactly the same transitions, we assume henceforth that $A$ has entries only $0$ or $1$. 
			 					For each $i \in D$ denote by $S(i)$ the set of allowed followers of $i$:
			 			\be
			 			S(i) = \{j \in D: A(i,j)=1\}.
			 			\en 
			 	We use the notations
			 			\be 
			 			|A_i|= \sum_{j=1}^d A(i,j), \quad s_A= \max_i |A_i|.
			 			\en 
			 			Thus we are interested in labelings $\lambda$ of the $k$-tree $\tau$ such that for all $x \in \tau$ and $g \in K$, we have $\lambda(xg) \in S(\lambda(x))$.
			 				Henceforth, unless stated otherwise, ``labeling" will mean  labeling that is valid, according to the restrictions imposed by the given transition matrix $A$.
			 				
			 						\begin{definition}\label{def:arrivalset}
			 					{For $v_0 \in \tau, i \in D,$ and $\mathcal T \subset \tau$, the {\em arrival set at $\mathcal T$ given $i$ at $v_0$} is
			 						the set of labelings of $\mathcal T$ that extend to valid (according to the restrictions imposed by the transition matrix $A$) labelings of the entire network that have label $i$ at $v_0$:
			 						\be 
			 						 \mathcal A(i,v_0,\mathcal T) = \{\lambda|_{\mathcal T}: \lambda \in X_A, \, \lambda(v_0)=i\}.
			 				\en }
			 				\end{definition}
			 				
			 		{\begin{definition}
			 			If $v_0 \in \tau$ is a vertex and $\mathcal T \subset \tau$ is a prospective ``target set", we say that the nonnegative transition matrix $A$ is 
			 			{\em $(v_0,\mathcal T$)-fair} if 
			 			$\mathcal A(i, v_0,\mathcal T)=\mathcal A(j, v_0,\mathcal T)$ for all $i,j \in D$. 
			 		\end{definition}}	
		 		
		 			\begin{definition}\label{def:P}
			 					For any $k \geq 1$ and $n \geq 0$, the set $P(k,n)$ consists of all the $d \times d$ nonnegative matrices $A$ that are $(\epsilon, L_n)$-fair, i.e. for all $i,j \in D$ we have $\mathcal A(i,\epsilon, L_n)=\mathcal A(j,\epsilon,L_n)$. \\
			 					The set $P^*(k,n)$ consists of those matrices that are $(\epsilon,L_n)$-complete.\\
			 					Further, $P(k)=\cup_n P(k,n)$, and $P^*(k)=\cup_nP^*(k,n)$.
			 					\end{definition}
			 					
			 			{Recall that our main question is how to determine whether or not a given matrix $A$ is in some $P(k,n)$ or not.}		
			 			Writing matrices on a single line for convenience, we {recall a few preliminary observations about these sets from \cite{PS2023}.} 
			 			
			 			\begin{enumerate}
			 				\item For all $k,n$, clearly $P^*(k,n) \subseteq P(k,n)$.
			 				\item If $A \in P(k,n)$ for some $k \geq 1$, then $A^n >0$, so that $A$ is primitive.
			 				\item $P(k+1,n) \subseteq P(k,n) \subseteq P(k, n+1)$.
			 				\item If $A^n >0$ and $A$ has a positive row, then $A \in P(k,n+1)$.
			 				\item The matrix $[011|100|010]$ is not in $P(2,n)$ for any $n$.
			 			\end{enumerate}
			 			
			 			{To these observations we can add a few more.}
			 			\begin{proposition}\label{prop:positiverow}
			 					Suppose that $d=|D| = k + 1$ and $A$ is a $d \times d$ $0,1$ primitive matrix.
			 					Then $A \in P^*(k)$ if and only	if $A$ has a positive row.
			 					\end{proposition}
			 			\begin{proof} 
			 				For $i=1, \dots,d$ we denote row $i$ of $A$ by $A_i=(A(i,1), \dots, A(i,d))$.
			 					{Suppose that $A$ is primitive, so that $A^p >0$ for some $p>0$ and there is an $i_0=1, \dots, d$ such that $A_{i_0}>0$ (meaning that all $A(i_0,j)>0, j=1, \dots, d$).
			 					{We claim that then $A \in P^*(k,p+1)$.} 
			 					Given $i \in D$, we can form a valid configuration on $\Delta_p$ by putting on each ray starting at the root a valid string that starts with $i$ and ends with $i_0$. 
			 					Then putting on $L_{p+1}$ any string in $D^{|L_{p+1}|}$ forms a valid configuration on $\Delta_{p+1}$.} 
			 				
			 				For the converse, assume that $A$ does not have a positive
			 				row, so that each row sum is less than or equal to $k$. 
			 				We have several cases:\\
			 				(1) Suppose first that the sum of each row is $k$, so that every symbol $i \in D$ has exactly $k$ successors that are allowed by $A$, in other words $|S(i)|=k$ for all $i \in D$. 
			 				Then we have two subcases. \\
			 				Case (1a): The $k + 1$ rows are distinct (meaning that no two are identical). 
			 				In this case all possible $0,1$ strings of length $d=k+1$ with exactly $k$ $1$'s appear among the rows of $A$, each one exactly once. 
			 				Stated differently, the sets $S(i), i \in D$, are distinct.
			 				Then $A$ is not in $P^*(k)$, because 
			 				 for each $i \in D$ we can form a {special configuration} $\tau_i$ on the $k$-tree, as follows. 
			 				Put $i$ at the root. 
			 				Proceeding inductively, at the $k$ nodes below any already labeled node put the $k$  allowed successors of that label, in any order. 
			 				Since in this situation each $k$-tuple from $D$ has a unique common predecessor, 
			 				the configuration $\tau_i$ has the property that for each $n \geq 1$ its configuration on $L_n$ forces the labeling of $L_{n-1}$ and therefore forces a single possible entry at the root.\\
			 				Case (1b): Two (or more) rows are identical. 
			 				In this case there is at least one $k$-tuple  $(x_1,\dots,x_k) \in D^k$ not all of whose entries are allowed successors of the same single symbol 
			 				(that is, $\{x_1, \dots, x_k\} \nsubseteq S(i)$ for any $i \in D$).
			 				 Then this $k$-tuple cannot appear in any $L_n$, so $A \notin P^*(k)$.\\
			 				(2) If at least one row sum is smaller than $k$, then again there is at least one $k$-tuple of elements of $D$ that cannot appear in any $L_n$, and therefore $A \notin P^*(k)$. 
			 				\end{proof}
			 			
			 			\begin{example}\label{ex:nopositive row}
			 				It can be shown that the matrix in Example \ref{ex:smallk} is in $P^*(2,2)$, although it does not have a positive row.
			 				\end{example}
			 				
			 				\begin{proposition}\label{prop:P*}
			 					{Let $k, n \geq 1$.
			 						The nonnegative matrix $A$ is in $P^*(k,n)$ if and only if $A \in P(k,n)$ and for each $k$-tuple $(a_1, \dots, a_k) \in D^k$ there is a common predecessor $i \in D$ (meaning that $A(i,a_j)>0$ for all $j=1, \dots,k$).}
			 					\end{proposition}
			 					\begin{proof}
			 						{Suppose that every $k$-tuple $a=(a_1, \dots, a_k) \in D^k$ has a common predecessor $p(a) \in D$. 
			 						Given an arbitrary configuration $\xi \in D^{L_n}$ on $L_n$, 
			 						for each site $x \in L_{n-1}$ for the $k$-tuple of labels of the children of $x$ on $L_n$ by hypothesis there is a common predecessor $i(x) \in D$. 
			 						Label the sites $x \in L_{n-1}$ with these symbols $i(x)$, and then repeat to produce a labeling of $L_{n-2}$ such that the transitions from $L_{n-2}$ to $L_{n-1}$ are consistent with the matrix $A$.  
			 						Finally we arrive at a valid labeling of $\Delta_n$, which has some symbol $i$ at the root. 
			 						If $A \in P(k,n)$, then given any $j \in D$ there is a valid labeling $\lambda'$ of $\Delta_n$ such that $\lambda'|L_n = \xi$. }
			 						
			 						{Conversely, suppose that $A \in P^*(k,n)$ for some $k,n \geq 1$. 
			 						Given any $k$-tuple $(a_1, \dots, a_k) \in D^k$, 
			 						define a labeling $\xi$ of $L_n$ by putting $\xi(xi) = a_i$ for each $x \in L_{n-1}$ and $i=1, \dots, k$. 
			 						By hypothesis, given any $j \in D$ there is a labeling $\lambda$ of $\Delta_n$ such that $\lambda|L_n = \xi$. 
			 						Then for each $x \in L_{n-1}$, $\lambda(x)$ is a predecessor of each $a_i, i=1, \dots, k$.}
			 						\end{proof}
			 					
			 		{\section{Relations between starting symbols}\label{sec:relations}		
			 			For this section we fix $k \geq 2$.} 
			 			Given a transition matrix $A$, we want to determine whether or not $A$ is $(\epsilon, L_n)$-fair for some $n \geq 1$, that is, whether or not $A \in P(k,n)$. 
			 			We now propose a methodical procedure for settling this question. 
			 			First, we establish some convenient terminology and notation. 
			 			We denote by $\Delta_n^a$ the subtree of the full $k$-tree with the label $a \in D$ assigned to the root. 
			 			\begin{definition}\label{def:relation}
			 			The notation $i \Rightarrow_n j$ means that for every labeling of $\Delta_n$ by an alphabet $D=\{1, \dots, d\}$ with $i$ at the root there is a labeling of $\Delta_n$ with $j$ at the root such that the two labelings agree on the $n$'th row, $L_n$. The notation $i \Rightarrow j$, read ``$i$ can be replaced by $j$",  means that there is an $n$ such that $i \Rightarrow_n j$.
			 			\end{definition}
			 				Thus $i \Rightarrow_n j$ means that $\mathcal A(i, \epsilon, L_n) \subseteq \mathcal A(j,\epsilon,L_n)$. 
			 				Any relation $i \Rightarrow_n j$ can be {\em realized on $\Delta_n$}, in the sense that given a valid (according to $A$) labeling $\lambda$ of $\Delta_n$ with $i$ at the root, it is possible, by trial and error if necessary, to find a (valid) labeling $\lambda'$ of $\Delta_n$ that has $j$ at the root and agrees with $\lambda$ on $L_n$.
			 			
			 			\begin{definition}\label{def:follower}
			 			Define the {\em $n$-follower set} (allowed by the transition matrix $A$) of a symbol $a \in D$ to be
			 			\be
			 			\mathcal F_n(a)=\{\lambda|_{L_n}: \lambda \text{ is an allowed labeling of } \Delta_n^a\}.
			 			\en
			 			\end{definition}
			 			Then 
			 			\be
			 			a \Rightarrow_n b 		 
			 			\text{ if and only if } \mathcal F_n(a) \subseteq \mathcal F_n(b).
			 			\en

			 			\begin{proposition}\label{prop:stay}
			 				If $i \Rightarrow_N j$, then $i \Rightarrow_n j$ for all $n \geq N$.
			 			\end{proposition}
			 			If $i \Rightarrow_N j$, we say that the relation {\em can be realized on} $\Delta_N$. Then, by Proposition \ref{prop:stay}, it can also be realized on $\Delta_n$ for all $n \geq N$.
			 		
			 			Denote by $A_i$ row $i$ of $A$. 			 			
			 			If $A_i \leq A_j$ (entry by entry), then $i \Rightarrow_1 j$. 
			 			Beyond that, sometimes we can tell that $i \Rightarrow j$ by ``moving" an entry of $A_i$ to produce a new row $A_i^*$ that can be compared successfully to $A_j$.
			 			
			 			Suppose that $a,b \in D$, $a \Rightarrow b$, $A_i(a)=1$, and $A_j(a)$=0, so that it is not the case that $A_i \leq A_j$. 
			 			Then we define $s_{ab}$ and $A_i^*=s_{ab}A_i$ by
			 			\be
			 			A_i^*(m)=(s_{ab}A_i)(m)=
			 			\begin{cases}
			 				A_i(m) &\text{ if } m \notin \{a,b\}\\
			 				0 &\text{ if } m=a \\
			 				1 &\text{ if } m=b.
			 			\end{cases}
			 			\en
			 			
			 			\begin{proposition}\label{prop:one}
			 				Suppose that $a,b \in D$, $a \neq b$, $a \Rightarrow_n b$, $A_i(a)=1$, and $A_j(a)=0$.
			 				If $A_i^*=s_{ab}A_i \leq A_j$, then $i \Rightarrow_{n+1} j$.
			 			\end{proposition}
			 			\begin{proof}
			 				Suppose that $a \Rightarrow_n b$. 
			 				Consider an allowed labeling $\lambda$ of $\Delta_{n+1}^i$ (with $i$ at the root). 
			 				Suppose that $c \in D$ is a label found on $L_1$ in the labeling $\lambda$, so that $A_i(c)=1$. 
			 				Because $A_i^*(b)=1$ and $A_i^* \leq A_j$, 
			 				{so that $A_j(b)=1$},
			 				if $c \in D \setminus \{a\}$ then $A_i(c) \leq A_j(c)$, and hence also $A_j (c)=1$. 
			 				Thus $c$ is allowed to follow $j$ at the root, and the labeling of the subtree $\Delta_n^c$ of height $n$, of $\Delta_{n+1}^i$, which has $c$ at its root, can be copied over to the corresponding subtree of $\Delta_{n+1}^j$. 
			 				
			 				If the symbol $a$ is found on $L_1$ in the labeling $\lambda$, then $a \Rightarrow_n b$ implies that the restriction of the labeling $\lambda$ to the subtree $\Delta_n^a$ of $\Delta_{n+1}^i$ can be replaced on the corresponding subtree $\Delta_n^b$ of $\Delta_{n+1}^j$ by one that has $b$ at its root and agrees with $\lambda$ on $L_n$. 
			 				(We again used that $A_i^*(b)=1$.)
			 				
			 				In this way the entire labeling $\lambda$ is modified, so that $i$ at the root is changed to $j$, and each entry (if any) of $a$ on the first row is changed to $b$. In particular the labeling of $L_{n+1}$ is the same on both $\Delta_{n+1}^i$ and $\Delta_{n+1}^j$.
			 			\end{proof}
			 			
			 			{
			 			\begin{remark}\label{rem:neq0} 
			 				The conclusion of Proposition \ref{prop:one} 
			 				holds also without the assumption that $A_j(a)=0$ (since $A_j(a)=1$ together with the other assumptions gives $A_i \leq A_j$).
			 			\end{remark}}

			 			The process described in Proposition \ref{prop:one} can be repeated many times.
			 			\begin{proposition}\label{prop:several}
			 				Suppose that for $m=1, \dots, r$ we have $a_m,b_m \in D$ with $a_m \neq b_m$, $a_m \Rightarrow_n b_m$, $A_i(a_m)=1$, and $A_i(c) \leq A_j(c)$ for all $c \notin \{a_m: m=1, \dots, r\}$. 
			 				Define
			 				\be
			 				{ \hat A_i(c)=
			 					\begin{cases}
			 						A_i(c) \text{ if } c \notin \{a_m, b_m: m=1, \dots , r\}\\
			 						0      \text{ if } c=a_m \text{ for some } m=1, \dots , r\\
			 						1      \text{ if } c=b_m \text { for some } m=1, \dots , r.
			 					\end{cases}
			 				}
			 				\en
			 				If $\hat A_i \leq A_j$, then $i \Rightarrow_{n+1} j$.
			 			\end{proposition}
			 			\begin{proof}
			 				If on row $1$ of a labeling $\lambda$ of $\Delta_{n+1}^i$ we have a symbol $c \in D \setminus \{a_m:m=1, \dots , r\}$, 
			 				since $A_i(c) \leq A_j(c)$ the symbol $c$ can also be written at the corresponding spot on $\Delta_{n+1}^j$ and the labeling of the subtree $\Delta_n^c$ of $\Delta_{n+1}^i$ can be carried over to the corresponding subtree of $\Delta_{n+1}^j$.
			 				
			 				Note that $\hat A_i(a_m)=0$ and $\hat A_i (b_m)=1$ for all $m=1, \dots ,r$.
			 				If on row $1$ of a labeling $\lambda$ of $\Delta_{n+1}^i$ we have a symbol $a_m$, since $a_m \Rightarrow_n b_m$ and 
			 				$\hat A_i \leq A_j$ implies
			 				$A_j(b_m)=1$, the labeling of the subtree $\Delta_n^{a_m}$ of $\Delta_{n+1}^i$ corresponds to a labeling of the corresponding subtree  $\Delta_{n}^{b_m}$ of $\Delta_{n+1}^j$ which agrees with the restriction of $\lambda$ to $L_{n+1}(\Delta_{n+1}^i)$. 
			 				Applying these relabelings for each $m=1, \dots ,r$ in turn, in whatever order, shows that $i \Rightarrow_{n+1} j$.
			 			\end{proof} 
			 			
			 			\begin{remark}
			 				We can think of starting with the vector $A_i$ and applying for each $m=1, \dots ,r$ in some order a ``move" that removes the entry $1$ from position $a_m$ and places it in position $b_m$. 
			 				The result is the vector $\hat A_i$. 
			 				Thus the ordering of the $(a_m, b_m)$ does not matter.
			 			\end{remark}
			 				With the notation of Definition \ref{def:follower} we can rewrite the proofs of Propositions \ref{prop:one} and Proposition \ref{prop:several}. 
			 			\begin{proposition} \label{prop:severalgen}. 
			 					Suppose that for $m=1, \dots ,r$ we are given $a_m,b_m \in D$ with  $\mathcal F_n(a_m) \subseteq \mathcal F_n(b_m)$ (i.e., $a_m \Rightarrow_n b_m$),  
			 				$A_i(c) \leq A_j(c)$ for all $c \notin \{a_m: m=1, \dots, r\}$, 
			 				{$A_i(a_m)=1$}
			 				and $A_j(b_m)=1$ (so that {$a_m \in S(i)$ and $b_m \in S(j)$).}
			 				Then $\mathcal F_{n+1}(i) \subseteq \mathcal F_{n+1}(j)$. 
			 			\end{proposition}
			 			\begin{proof}
			 				Let $\lambda$ be a labeling of $\Delta_{n+1}$ with $i$ at the root. 
			 				Replace $i$ at the root by $j$ and each entry (if any) of any $a_m$ on $L_1$ by the corresponding $b_m$.
			 				Label each $\Delta_n^{b_m}$ with an allowed labeling that agrees on its $L_n$ with $\lambda$ restricted to $L_n(\Delta_n^{a_m})$.
			 				Since $A_j(b_m)=1$ and
			 				$\mathcal F_n(a_m) \subset \mathcal F_n(b_m)$, the new labeling is a legal labeling of $\Delta_{n+1}$ with $j$ at the root which agrees on $L_{n+1}$ with $\lambda$. 
			 			\end{proof}	 			 
			 			
			 			\begin{remark} 			 
			 				If there aren't any $a_m$, we still conclude that $i \Rightarrow j$, since then $A_i \leq A_j$ (without using the assumption that all $A_j(b_m)=1$, which in the $A^*$ approach follows from $A_i^* \leq A_j$).
			 			\end{remark}

			 			Having applied Proposition \ref{prop:several}, we may have discovered new relations $i \Rightarrow j$ that were not known before. 
			 			If the original set of relations $a_m \Rightarrow b_m$ is now larger, we may apply Proposition \ref{prop:several} to the new set. 
			 			In the next section we describe a systematic procedure for building up the known set of relations, and 
			 			in Section \ref{sec:works} we will prove that {\em all} relations are found this way.

			 			\section{an algorithm for finding relations}\label{sec:procedure}
			 			We describe an algorithm for adding to a list of known relations. 
			 			\begin{definition}\label{def:degree}
			 			A relation $i \Rightarrow j$ is said to be of {\em degree $n$} if $n$ is the smallest $m$ such that $i \Rightarrow_m j$, i.e., the smallest $m$ such that for every allowed labeling $\lambda$ of $\Delta_m$ that has $i$ at the root there is a labeling of $\Delta_m$ that has $j$ at the root and agrees on the last row $L_m$ with $\lambda$. 
			 			We denote by $\mathcal D_n$ the set of relations of degree $n$. 
			 			\end{definition}
			 				The relations $i \Rightarrow i, i \in D$, are of degree $0$, and
			 				the relations $i \Rightarrow_1 j$, equivalent to $A_i \leq A_j$, are of degree $1$.

			 				The algorithm consists of a sequence of {\em rounds}, beginning with Round $0$.
			 					For each $n \geq 0$, we denote by $\mathcal K_n$ the set of relations that are {\em known at time $n$} (that is, after Round $n$), defined as follows. 
			 			 			We use an evolving $d \times d$ matrix $R_n$ to keep track of relations as they are found. 
			 			To begin, in Round $0$ set $R_0(i,j)=1$ if $i=j$, otherwise $R_0(i,j)=0$. 
			 			The relations $i \Rightarrow i, i \in D$, constitute $\mathcal K_0$ and are said to have {\em height} $0$.
			 			
			 			If $n \geq 0$ and at Round $n+1$ we find, by the process below, that $R_n(i,j)=0$ but $i \Rightarrow_{n+1} j$, then we set $R_{n+1}(i,j)=n+1$ and add the relation $i \Rightarrow_{n+1} j$ to our collection of known relations, otherwise we leave $R_{n+1}(i,j)=0$. 
			 			In this case we say that the relation $i \Rightarrow j$ is {\em discovered at Round $n+1$} and has {\em height} $n+1$.
			 						 			
			 		The following describes the action in the succession of rounds, having started with Round $0$ described above. 
			 			
			 			{Assume now that $n \geq 1$ and {certain} relations $a \Rightarrow_n b$ are known, i.e. we know the set $\mathcal K_n$.  
			 				We proceed to conduct Round $n+1$. 
			 				Begin with $\mathcal K_{n+1} = \mathcal K_n$.
			 				Fix $i,j \in D$, and suppose $(i \Rightarrow j) \notin \mathcal K_n$. 
			 				Check all $a \in S(i), b \in S(j)$, to see whether 
			 				{it is known that} $a \Rightarrow b$, i.e. whether or not $(a \Rightarrow b) \in \mathcal K_n$, equivalently, whether $R_n(a,b)=1$.

			 					Let $\{(a_m,b_m) \in D \times D: m=1, \dots,r\}$ denote the set of all $(a,b) \in D \times D$ for which $\mathcal F_n(a_m) \subseteq \mathcal F_n(b_m)$ (i.e., $a_m \Rightarrow_n b_m$),  
			 				$A_i(c) \leq A_j(c)$ for all $c \notin \{a_m: m=1, \dots, r\}$, 
			 				{$A_(a_m)=1$,}
			 				and $A_j(b_m)=1$ (so that {$a_m \in S(i)$ and} $b_m \in S(j)$).
			 				Then by Proposition \ref{prop:severalgen} or \ref{prop:several},			 				
			 				$\mathcal F_{n+1}(i) \subseteq \mathcal F_{n+1}(j)$, 
			 				so we add the relation $(i \Rightarrow j)$ to $\mathcal K_{n+1}$  
			 				and set $R_{n+1}(i,j)=n+1$. 
			 			
			 				Thus in Round $1$ we set $R_1(i,j)=1$ if and only if $A_i \leq A_j$, i.e. $i \Rightarrow_1 j$. 
			 				
			 					{The following Proposition summarizes the preceding results in a way that is especially convenient for implementing the procedure by computer.
			 				\begin{proposition}\label{prop:sum}
			 					Let $n \geq 0$.
			 					The relation $i \Rightarrow_{n+1} j$ is known after Round $n+1$ if 
			 				for every $a \in S(i)$ there is $b \in S(j)$ such that $R_n(a,b)>0$. 
			 				The converse holds if the tree dimension 
			 				$k \geq s_A =\max_i |A_i|$. 
			 				\end{proposition}}
			 				\begin{proof}
			 				{Suppose that $S(i)=\{a_1,\dots,a_p\}$} and for every $m=1, \dots, p$ there is $b_m \in S(j)$ such that $R_n(a_m,b_m)>0$. 
			 					Then $a_m \Rightarrow_n b_m$, $a_m \in S(i)$, and $b_m \in S(j)$ for all $m$.
			 					And for all $c \notin S(i)$, $0 = A_i(c) \leq A_j(c)$. 
			 					Then by Proposition \ref{prop:severalgen} or \ref{prop:several}, $i \Rightarrow_{n+1}j$.
			 					
			 					{For the converse, see Proposition \ref{prop:n} (1).}
			 					\end{proof}

			 					\begin{remark}\label{rem:suff}
			 					{Propositions \ref{prop:several} and \ref{prop:severalgen} show that if a relation $i \Rightarrow j$ is discovered on round $n+1$, then it can be realized on $\Delta_{n+1}$ (meaning that $i \Rightarrow_{n+1} j$, i.e., for any labeling of $\Delta_{n+1}$ with $i$ at the root there is a labeling of $\Delta_{n+1}$ with $j$ at the root such that the two labelings agree on $L_{n+1}$).}
			 					{In other words, the height of any relation is greater than or equal to its degree.}
			 					\end{remark}
			 					The following examples show how the algorithm is carried out in practice, by a computer.
			 			
			 			\begin{example}\label{ex:ex1}

			 				We apply the algorithm to the matrix $A=[110|001|100]$ that specifies the allowed transitions for labeling the $2$-tree by an alphabet of $3$ symbols. 
			 				Note that that the tree dimension $k=2$ is greater than or equal {to} the maximum row sum of the transition matrix, $s_A=2$.
			 				
			 				\noindent
			 				Round 0: Set $R_0=I=[100|010|001]$.
			 				
			 				\noindent
			 				{Round 1:} Look for basic relations: $A_i\leq A_j$. 
			 				We see that $i=3$ and $j=1$ are the only
			 				pair satisfying this relation. So update the relation matrix to
			 				$R_1=[100|010|101]$. 
			 				Check that 
			 				some newly discovered relations were added at this round. If so continue,
			 				otherwise stop.
			 				
			 				{ Note:} $3 \Rightarrow 1$ has height $1$.
			 				This means that
			 				the replacement of $3$ by $1$ can be realized on $\Delta_1$ (but not on $\Delta_0$).  
			 				Also, since it cannot be realized on $\Delta_0$, this relation has degree $1$.
			 				
			 				\noindent
			 				{Round 2:}	 Check to see if we can find some $A_i^* \leq A_j$. 
			 				If so,
			 				we can find such $A_i^*$ then we update $R(i,j)=1$. 
			 				Here we have only
			 				$R_1(3,1)=1$.
			 				Then $s_{31}(A_2)=s_{31}([001])=[100]=A_2^*$. 
			 				Thus, we have
			 				$A_2^*\leq A_1$ and $A_2^*\leq A_3$, so that 
			 				$2 \Rightarrow 1$ and $2 \Rightarrow 3$. 
			 				These are the only two new relations that can be discovered using our current $R$. 			 				
			 			Update the relation matrix to
			 				$R_2=[100|212|101]$. 
			 					Check that 
			 				some newly discovered relations were added at this round. If so continue,
			 				otherwise stop.
			 						 				
			 				{ Note:} The two added relations ($2 \Rightarrow 1, 2 \Rightarrow 3$) have height and degree $2$. 
			 				
			 				\noindent
			 				{Round 3:} For each $i,j$ such that $R_2(i,j)=0$ (we use only relations
			 				provided by $R_2$, relations newly discovered relations in this round 
			 				cannot be used in this round), check to see if we can find some $A_i^*$
			 				such that $A_i^*\leq A_j$. 
			 				Here, we can
			 				discover only one relation. 
			 				So we use $s_{21}(A_1)=s_{21}([110])=[100]=A_1^*$, and
			 				we get $A_1^*\leq A_3$, yielding $1 \Rightarrow 3$. 
			 				Hence we update the relation matrix to
			 				$R_3=[103|212|101]$. 
			 					Check that 
			 				some newly discovered relations were added at this round. If so continue,
			 				otherwise stop.
			 				
			 				{Note:} The newly added relation ($1 \Rightarrow 3$) is of height and degree $3$. We used
			 				a relation of degree $2$ ($2 \Rightarrow 1$) to reach this conclusion.
			 				
			 				\noindent
			 				{Round 4:} We have $R_3=[103|212|101]$.
			 				 For each $i,j$ such that $R_3(i,j)=0$, using only relations
			 				provided by $R_3$, check to see if we can find some $A_i^*$ with $A_i^*\leq A_j$.
			 				Here we have two cases left.\\
			 				(1) For $R_3(1,2)=0$, we use
			 				$s_{13}s_{23}(A_1)=s_{13}s_{23}([110])=s_{13}([101])=[001]=A_1^*$, and $A_1^*\leq A_2$,
			 				hence we conclude $1 \Rightarrow 2$. \\
			 			(2) For $R_3(3,2)=0$: $s_{13}(A_3)=s_{13}([100])=[001]=A_3^*$,
			 				and $A_3^*\leq A_2$ , so we conclude $3 \Rightarrow 2$.
			 				
			 				Update the relation matrix to
			 				$R_4=[143|212|141]$.
			 					Check that 
			 				some newly discovered relations were added at this round. If so continue,
			 				otherwise stop.
			 				
			 				{ Note:} Both relations added in Round 4 are of height and degree $4$ (we used a relation of degree $3$ 
			 				($1 \Rightarrow 3$) to reach this conclusion).
			 				
			 				\noindent
			 				{Round 5:} In Round five no relations can be added,
			 				{since $R_4$ is a positive matrix.} Therefore the computations must stop.
			 				
			 				{Note:} The final matrix is $R_4=[143|212|141]$. 
			 				Ignoring
			 				diagonal elements, for each $i,j$ the entry $R_4(i,j)$ is the degree of the relation $i \Rightarrow j$. 
			 				For example,
			 				$R_4(3,1)=1$, indicating that $3 \Rightarrow 1$ is of degree $1$, and $R_4(1,2)=4$ indicating
			 				that $1 \Rightarrow 2$ is of degree $4$.
			 				
			 				{Note:} The algorithm yields the conclusion that $A\in P(k,n)$ for $k\geq 2$ and $n\geq 4$.
			 				\end{example}
			 			
			 				\begin{example}\label{ex:ex2}
			 				Now consider the transition matrix  $A=[1001|1000|0100|0010]$.

			 				\noindent
			 				{Round 0:} Set $R_0=I=[1000|0100|0010|0001]$.
			 				
			 				\noindent
			 				{Round 1:} For all $i,j$ with $R_0(i,j)=0$, check if $A_i\leq A_j$.
			 				Here, we have only one case, $A_2\leq A_1$,
			 				hence, update the relation matrix to 
			 				$R_1=[1000|1100|0010|0001]$. Since a relation was added, continue computations.
			 				
			 				\noindent
			 				{Round 2:} For all $i,j$ such that $R_1(i,j)=0$, using relations as provided
			 				by $R_1$ (only), check to see if we can find $A_i^*$ such that $A_i^*\leq A_j$.
			 			We can find two cases: $s_{21}(A_3)=s_{21}([0100])
			 				=[1000]=A_3^*$, so that $A_3^*\leq A_1$ and $A_3^*\leq A_2$. 
			 				 we update
			 				 the relation matrix to 
			 				$R_2=[1000|1100|2210|0001]$. 
			 				Since some relations were added, continue computations.
			 				
			 				\noindent
			 				{Round 3:} For all $i,j$ such that $R_2(i,j)=0$, using relations as
			 				provided by $R_2$ (only), check to see if we can find $A_i^*$ such
			 				that $A_i^*\leq A_j$. 
			 				 Here, we have\\
			 			(1) 	$s_{31}(A_4)=s_{31}([0010])=[1000]=A_4^{*(1)}$, and $A_4^{*(1)}\leq
			 				A_1,A_2$, so $4 \Rightarrow 1$ and $4 \Rightarrow 2$.\\
			 		(2)	We also have
			 				$s_{32}(A_4)=s_{32}([0010])=[0100]=A_4^{*(2)}$, and $A_4^{*(2)}\leq A_3$, so
			 				$4 \Rightarrow 3$.
			 				  Update the relation matrix to
			 				$R_3=[1000|1100|2210|3331]$. Since some relations were added, continue computations.
			 			
			 				\noindent
			 				{Round 4:} For all $i,j$ such that $R_3(i,j)=0$, using relations 
			 				provided by $R_3$ (only), check to see if we can find $A_i^*$ such
			 				that $A_i^*\leq A_j$. 
			 				Here, we have 
			 				only one case: $s_{41}(A_1)=s_{41}([1001])=[1000]=A_1^*$, and
			 				$A_1^*\leq A_2$, showing that $1 \Rightarrow 2$. 
			 				Hence we 
			 				update the relation matrix to
			 				$R_4=[1400|1100|2210|3331]$. Since a relation was added, continue computations.
			 				
			 				\noindent
			 				{Round 5:} For all $i,j$ such that $R_4(i,j)=0$, using relations
			 				as provided by $R_4$ (only), check to see if we can find $A_i^*$ such that
			 				$A_i^*\leq A_J$. 
			 				Here we have two cases:\\
			 			(1)	$s_{12}s_{42}(A_1)=s_{12}([1100])=[0100]=A_1^*$, $A_1^*\leq A_3$, so $1 \Rightarrow 3$.\\
			 				(2) Also, $s_{12}(A_2)=s_{12}([1000])=[0100]=A_2^*$,
			 				$A_2^*\leq A_3$, so $2 \rightarrow 3$.
			 				Update the relation matrix to
			 				$R_5=[1450|1150|2210|3331]$. Since some relations were added, continue computations.
			 				
			 				\noindent
			 				{Round 6:} For all $i,j$ such that $R_5(i,j)=0$, using relations
			 				as provided by $R_5$ (only), check to see if we can find $A_i^*$
			 				such that $A_i^*\leq A_j$. 
			 				 Here, we
			 				have three cases: \\
			 				(1) $s_{13}s_{43}(A_1)=r_{13}s_{43}([1001])=s_{13}([1010])=[0010]=A_1^*$,
			 				and $A_1^*\leq A_4$, hence update $1 \Rightarrow 4$. \\
			 				(2) Also, $s_{13}(A_2)=r_{13}([1000])=
			 				[0010]=A_2^*$, and $A_2^*\leq A_4$, hence update $2 \Rightarrow 4$.\\
			 				(3) Finally, $s_{23}(A_3)=
			 				s_{23}([0100])=[0010]=A_3^*$, and $A_3^*\leq A_4$, so update $3 \Rightarrow 4$.
			 				
			 				 Update the relation matrix to
			 				$R_6=[1456|1156|2216|3331]$. Since some relations were added, continue computations.
			 				
			 				\noindent
			 				{Round 7:} For all $i,j$ such that $R_6(i,j)=0$, using relations as
			 				provided by $R_6$ (only), 
			 				{check to see if we can find $A_i^*$ and $A_j$ such that $A_i^* \leq A_j$,} but we have no such $i,j$.
			 				 Since no relations can be added, stop computations.
			 				
			 				{ Note:} The degree matrix is given by
			 				\be
			 				R_6=[1456|1156|2216|3331].
			 				\en
			 				{ Conclusion:} 
			 				$A\in P(2,6)$, that is, $A\in P(k,n), k\geq 2, n\geq 6$.
			 			\end{example}
			 			\begin{example}\label{ex:ex3}
			 				Now we investigate the transition matrix $A=[0111|1000|0100|0010]$.
			 				
			 				{Round 0:} Set $R_0=I$.
			 				
			 				\noindent
			 				{Round 1:} For all $i,j$ such that $R_0(i,j)=0$, check to see if
			 				$A_i\leq A_j$. Here, we have two cases:
			 				$A_3\leq A_1$ and $A_4\leq A_1$. We update $3 \Rightarrow 1$ and $4 \Rightarrow 1$.
			 				
			 				 Update the relation matrix to
			 				$R_1=[1000|0100|1010|1001]$. Since some relations were added, continue computations.
			 				
			 				\noindent
			 				{ Round 2:} For all $i,j$ such that $R_1(i,j)=0$, check to see if
			 				we can find some $A_i^*$ such that $A_i^*\leq A_j$. 
			 			s We have (only) one case: $r_{31}(A_4)=r_{s1}([0010])
			 				=[1000]=A_a^*$, and $A_4^*\leq A_2$. Hence, update $4 \Rightarrow 2$.
			 				
			 				{Update $R_2=[1000|0100|1010|1201]$.}
			 				Since a relation was added, continue computations.
			 				
			 				\noindent
			 				{Round 3:} For all $i,j$ such that $R_2(i,j)=0$, check to see 
			 				if we can find $A_i^*$ such that $A_i^*\leq A_j$. We have no such case.
			 				 Since no relations are added,
			 			 stop computations.
			 				
			 				{Conclusion: $A \notin P(3,2)$; hence for all $k \geq 3$, $A \notin P(k,2)$. 
			 				By Theorem \ref{thm:disc}, for every $k \geq 3, n \geq 1$, we have $A \notin P(k,n)$. 
			 					But $A \in P(1,7)$, since $A^7>0$. ($A$ is primitive.)}
			 				
			 				\end{example}

			 			\section{The algorithm works}\label{sec:works}
			 			{Let $s_A$ denote the maximum row sum of the transition matrix $A$:
			 				\be
			 				s_A= \max \{|A_i|: i=1 , \dots, d\}.
			 				\en }
			 				We aim to show that if $k \geq s_A$, then the algorithm described above will actually produce all relations that exist among the symbols in $D$. 
			 				{\begin{definition}\label{def:full}
			 						A labeling $\lambda$ of any subtree $\tau$ of the $k$-tree is {\em full} if for each site $x \in \tau$ all allowed followers of the label assigned to $x$ by $\lambda$ appear among the labels of the successors of $x$ in the tree:
			 						\be
			 						\{\lambda(xg): g \in K\} = S(\lambda(x)).
			 						\en
			 					\end{definition}
			 				}
			 				{We will show first that} every labeling of the tree results by ``reducing" a full labeling and keeping track of the configurations on the target row, $L_{n}$

			 				We define two operations, {\em Switch} $S$ and {\em Replace} $R$ on a labeling $\Lambda$ of a finite tree $\Delta_n^i$ of height $n$ with $i$ at the root (on line $L_0$). 
			 						 				
			 				Assume that $1 \leq r \leq n$ and $x,y \in L_r$ with $x=zu$ and $y=zv$ for a site $z \in L_{r-1}$, some $u,v \in K$. 
			 				Then we define the {\em Replacement} $\tilde \Lambda=R(x,y)(\Lambda)$ of the labeled subtree under $x$ by the labeled subtree under $y$ by defining $\tilde \Lambda(xw)=\Lambda(yw)$ for every word $w$ on $K$ with $0 \leq |w| \leq n-|x|$, and $\tilde\Lambda(t)=\Lambda(t)$ for all other $t \in \Delta_n$.
			 				
			 				We define the {\em Switch} $\overline \Lambda=S(x,y) (\Lambda)$ of the labeled subtrees under $x$ and $y$ by 
			 				$\overline \Lambda(xw)=\Lambda(yw)$ and $\overline \Lambda(yw)=\Lambda(xw)$ 
			 				for every word $w$ on $K$ with $0 \leq |w| \leq n-|x|$, and $\overline \Lambda(t)=\Lambda(t)$ for all other $t \in \Delta_n$.
			 				
			 				Note that all this is legal because $\Lambda(x)$ and $\Lambda(y)$ are both allowed followers of $\Lambda(z)$ according to the transition matrix $A$.
			 				
			 				\begin{lemma}\label{lem:switch}
			 					Assume that $k \geq s_A$. 
			 					Then {\em every} labeling $\Lambda$ of $\Delta_n^i$ results from a {\em full} labeling $\Lambda^*$ of $\Delta_n^i$ by applying a sequence of replacements followed by a sequence of switches.
			 				\end{lemma}
			 				\begin{proof} 
			 					Let a labeling $\Lambda$ and a full labeling $\Lambda^*$ of $\Delta_n^i$ be given. 
			 					The $k$ sites on $L_1$ are ordered as $x_1, \dots, x_k$ according to the natural ordering of $K=\{1, \dots, k\}$ (and similarly the sites on each $L_r$ are ordered according to the lexicographic ordering on the set of strings on $K$ that define them). 
			 					So we regard $\Lambda^*(L_1)$ and $\Lambda(L_1)$ as strings on $D$.
			 					
			 					Reducing the alphabet if necessary, we may assume that all symbols of $D$ appear in $\Lambda(L_1)$ and $\Lambda^*(L_1)$. (Use Replace to move any subtrees under extra symbols in $\Lambda^*(L_1)$ (that do not appear in $\Lambda(L_1)$) to subtrees under any symbol that appears in $\Lambda(L_1)$.)
			 					Applying replacements, we may also assume that each symbol $a \in D$ appears exactly once in the string $\Lambda^*(L_1)$, except that $a=1$ appears $m=k-d+1$ times (the required rest of the number of times): 
			 					$|\{r \in K:\Lambda^*(x_r) = a\}|=1$ for all $a \neq 1$.
			 					
			 					Then we can use Replace to redistribute the subtrees under the symbol $1$ on $L_1$ so that the strings $\Lambda^*(L_1)$ and $\Lambda(L_1)$ have the same symbol counts: for every $a \in D$, $|\{r \in K: 	\Lambda^*(x_r)=a\}| = |\{r\in K: \Lambda(x_r)=a\}|$.	
			 					
			 					Then apply Switch enough times to the current image of $\Lambda^*(L_1)$ to make it identical to $\Lambda(L_1)$.	 					
			 					
			 					Repeat for	
			 					$r=2, \dots, n$ in turn to obtain the image of $\Lambda^*$ after replacements and switches that agrees with $\Lambda$ on $\Delta_n$.			 					
			 					
			 				\end{proof}
			 				
			 				\begin{lemma}\label{lem:same}
			 					If labelings $\Lambda$ and $\Lambda'$ of $\Delta_n$ agree on $L_n$, then so do their images after the same Replace or Switch is applied to them.
			 				\end{lemma}
			 				\begin{proof}
			 					$\Lambda(L_n) = \Lambda'(L_n)$ means that all of their corresponding substrings are equal. 
			 					When we apply the same Replace or Switch operation to corresponding pairs of substrings in each, we end up with the two strings still being equal.
			 					
			 					More precisely, 
			 					suppose that $r \geq 1$, $x,y \in L_r \subseteq \Delta_n$, $x=zu, y=zv$ for some $z \in L_{r-1}, u,v \in K$, and we apply $R(x,y)$ or $S(x,y)$ to each of $\Lambda$ and $\Lambda'$, producing $\tilde \Lambda$ and $\tilde \Lambda'$, respectively.
			 					Denote by $\Delta_{n-r}(x)$ the subtree of height $n-r$ under the site $x$, and by $\Delta_{n-r}(y)$ the subtree of height $n-r$ under $y$, and by $L_{n-r}(x), L_{n-r}(y)$ their respective final rows. 
			 					Since $\Lambda$ and $\Lambda'$ agree on $L_n$, they agree on $L_{n-r}(x)$, and they agree on $L_{n-r}(y)$. 
			 					After the replacement or switch is applied, $\tilde \Lambda (L_{n-r}(x))=\tilde \Lambda'(L_{n-r}(x))$, and $\tilde \Lambda (L_{n-r}(y))= \tilde \Lambda'(L_{n-r}(y))$. 
			 					Since no other sections of the strings $\Lambda(L_n)$ and $\Lambda'(L_n)$ were changed by the replacement or switch, $\tilde \Lambda(L_n)= \tilde \Lambda'(L_n)$.
			 				\end{proof}

			 				Now we are ready to state and prove the key result that guarantees the efficacy of the relation-finding algorithm. 
			 			\begin{proposition}\label{prop:n}
			 				{1. Assume that $k \geq s_A$.
			 					If $n \geq 1$ and $i \Rightarrow_{n+1}j$, then {for every $a \in S(i)$ there is $b=b(a) \in S(j)$ such that $a \Rightarrow_n b$.}} \\
			 				{2. Conversely, suppose that $n \geq 1$, $i, j \in D$, and for every $a \in S(i)$ there is $b=b(a) \in S(j)$ (possibly $b(a)=a$) such that $a \Rightarrow_n b(a)$. 
			 					Then $i \Rightarrow_{n+1}j$.}
			 			\end{proposition}
			 			\begin{proof}
			 				Let $\Lambda^i_{n+1}$ be a ``full" labeling of $\Delta^i_{n+1}$ (so that the root has label $i \in D$), meaning that at each site {$x \in \Delta^i_{n+1} \setminus L_{n+1}$} the followers $xf, f \in \{1, \dots, k\}$, are labeled, among them, with {\em all} the followers of the label of $x$ that are allowed by the transition matrix $A$:
			 				\be
			 				\{\Lambda(xf): f \in \{1, \dots, k\}\} = \{b \in D: A(\Lambda(x),b)=1\}.
			 				\en 
			 				Since $i \Rightarrow_{n+1} j$, there is an allowed labeling $\Lambda^j_{n+1}$ of $\Delta^j_{n+1}$ that has $j$ at the root and agrees on $L_{n+1}$ with $\Lambda^i_{n+1}$.

			 				
			 				{Let $x \in L_1$ be a site in $\Delta^i_{n+1}$ that is occupied by a symbol 
			 					$\Lambda^i_{n+1}(x)=a \in D$. If $a \in S(j)$, we may take $b(a)=a$ and we will be done. So assume that 
			 					$a \in S(i) \setminus S(j)$. 
			 					Denote by $b=b(a)$ the label of site $x$ assigned by $\Lambda^j_{n+1}$, i.e. $b=\Lambda^j_{n+1}(x)$. 
			 					We claim that $a \Rightarrow_n b$.}
			 				
			 				Let $\lambda^*=\Lambda_{n+1}^i|\Delta_n(x)$ denote the restriction of the full labeling $\Lambda_{n+1}^i$ to the subtree $\Delta_n(x)$, and 
			 				let $\lambda$ be any allowed labeling of the subtree $\Delta_n(x)$ {that has $a$ at the root}. 
			 				We aim to show that there is an allowed labeling $\lambda'$ of $\Delta_n(x)$ that has $b$ at the root and agrees on the last row $L_n$ of $\Delta_n(x)$ with $\lambda$. 
			 				
			 				Let {$\lambda_1=\Lambda^j_{n+1}|\Delta_n(x)$. 
			 					Then $\lambda^*$ and $\lambda_1$ agree on $L_n$.
			 					By Lemma \ref{lem:switch}, $\lambda$ can be produced by applying a sequence of replacements and switches, starting with $\lambda^*$. 
			 					Let us apply simultaneously the same sequence of operations to $\lambda_1$. 
			 					By Lemma \ref{lem:same}, after each operation, the new labelings continue to agree on $L_n$. 
			 					At the end we arrive at the labeling $\lambda$ on $\Delta_n^a$ and some labeling $\lambda'$ on $\Delta_n(x)$, which still has label $b(a)$ at its root, and these two labelings agree on $L_n$. 
			 					See Figure \ref{fig:switch} for an illustration of the sequence of moves in a particular case.

		{2. Conversely, assume that for every $a \in S(i)$ there is $b=b(a) \in S(j)$ (possibly 
		$b(a)=a$) such that $a \Rightarrow_n b(a)$. 
		Let $\Lambda^i_{n+1}$ be a labeling of $\Delta^i_{n+1}$ (with $i$ at the root). 
		For every site $x \in L_1$, with label $a(x)=\Lambda^i_{n+1}(x)$, by hypothesis there is $b=b(a(x)) \in S(j)$ (possibly $b(a(x))=a(x)$) such that $a(x) \Rightarrow_n b(a(x))$. 
		Then for every $x \in L_1$ the labeling $\lambda = \Lambda^i_{n+1}|\Delta_n^{a(x)}$ has a corresponding labeling $\lambda''$ of $\Delta_n^{b(a(x))}$, also with root at site $x \in L_1$, that agrees with $\lambda$ on row $n$ of $\Delta_n^{a(x)}$.}
		
		{The labeled subtrees $\Delta_n^{b(a(x))}, x \in L_1$, fill up $\Delta^j_{n+1}$ and so determine a labeling $\Lambda^j_{n+1}$ of $\Delta^j_{n+1}$ that agrees on $L_{n+1}$ with $\Lambda^i_{n+1}$.}
		
		Note that (2) does not require $k \geq s_A$.
		}	\end{proof}
			
			 					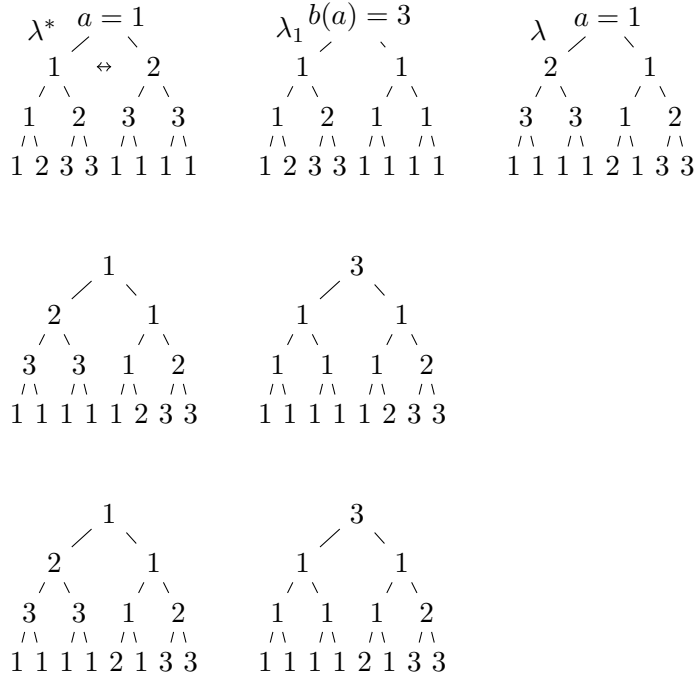
\begin{figure}[h]\label{fig:switch}
			 						\centering
			 					\begin{tikzpicture} [scale=.33]
			 					\node at (1,25.5) (lambda*) {$\lambda^*$}; \node at ( 11,25.5) (lambda1) {$\lambda_1$}; \node at ( 21,25.5) (lambda) {$\lambda$};
			 					
	\node at (3.8,26) (a) {$a=1$}; \node at (13.8,26) (3) {$b(a)=3$}; \node at (23.8,26) (1) {$a=1$};
	
	\node at (1.5,24) (u1) {$1$};\node at (5.5,24) (u2) {$2$}; \node at (11.5,24) (u3) {$1$}; \node at (15.5,24) (u4) {$1$}; \node at (21.5,24) (u5) {$2$}; \node at (25.5,24) (u6){$1$};
	
	\node at (.5,22) (x1) {$1$}; \node at (2.5,22) (x2) {$2$}; \node at (4.5,22) (x3) {$3$}; \node at (6.5,22) (x4) {$3$};
	\node at (10.5,22) (x5) {$1$};\node at (12.5,22) (x6) {$2$}; \node at (14.5,22) (x7) {$1$}; \node at (16.5,22) (x8) {$1$};
	\node at (20.5,22) (x9) {$3$}; \node at (22.5,22) (x10) {$3$}; \node at (24.5,22) (x11) {$1$}; \node at (26.5,22) (x12) {$2$};
	
	\node at (0,20) (y1) {$1$}; \node at (1,20) (y2) {$2$}; \node at (2,20) (y3) {$3$}; \node at (3,20) (y4) {$3$};
	\node at (4,20) (y5) {$1$}; \node at (5,20) (y6) {$1$}; \node at (6,20) (y7) {$1$}; \node at (7,20) (y8) {$1$};
			 						
			\node at (10,20) (z1) {$1$}; \node at (11,20) (z2) {$2$}; \node at (12,20) (z3) {$3$}; \node at (13,20) (z4) {$3$};
		\node at (14,20) (z5) {$1$}; \node at (15,20) (z6) {$1$}; \node at (16,20) (z7) {$1$}; \node at (17,20) (z8) {$1$};	 				
		
			\node at (20,20) (w1) {$1$}; \node at (21,20) (w2) {$1$}; \node at (22,20) (w3) {$1$}; \node at (23,20) (w4) {$1$};
		\node at (24,20) (w5) {$2$}; \node at (25,20) (w6) {$1$}; \node at (26,20) (w7) {$3$}; \node at (27,20) (w8) {$3$};			
			\node at (3.7,16) (a') {$1$}; \node at (13.7,16) (3') {$3$}; 
		
		\node at (1.5,14) (u1') {$2$};\node at (5.5,14) (u2') {$1$}; \node at (11.5,14) (u3') {$1$}; \node at (15.5,14) (u4') {$1$}; 
		
		\node at (.5,12) (x1') {$3$}; \node at (2.5,12) (x2') {$3$}; \node at (4.5,12) (x3') {$1$}; \node at (6.5,12) (x4') {$2$};
		\node at (10.5,12) (x5') {$1$};\node at (12.5,12) (x6') {$1$}; \node at (14.5,12) (x7') {$1$}; \node at (16.5,12) (x8') {$2$};

		\node at (0,10) (y1') {$1$}; \node at (1,10) (y2') {$1$}; \node at (2,10) (y3') {$1$}; \node at (3,10) (y4') {$1$};
		\node at (4,10) (y5') {$1$}; \node at (5,10) (y6') {$2$}; \node at (6,10) (y7') {$3$}; \node at (7,10) (y8') {$3$};
		
		\node at (10,10) (z1') {$1$}; \node at (11,10) (z2') {$1$}; \node at (12,10) (z3') {$1$}; \node at (13,10) (z4') {$1$};
		\node at (14,10) (z5') {$1$}; \node at (15,10) (z6') {$2$}; \node at (16,10) (z7') {$3$}; \node at (17,10) (z8') {$3$};	 
							\node at (3.7,6) (a'') {$1$}; \node at (13.7,6) (3'') {$3$}; 
					
					\node at (1.5,4) (u1'') {$2$};\node at (5.5,4) (u2'') {$1$}; \node at (11.5,4) (u3'') {$1$}; \node at (15.5,4) (u4'') {$1$}; 
					
					\node at (.5,2) (x1'') {$3$}; \node at (2.5,2) (x2'') {$3$}; \node at (4.5,2) (x3'') {$1$}; \node at (6.5,2) (x4'') {$2$};
					\node at (10.5,2) (x5'') {$1$};\node at (12.5,2) (x6'') {$1$}; \node at (14.5,2) (x7'') {$1$}; \node at (16.5,2) (x8'') {$2$};

					\node at (0,0) (y1'') {$1$}; \node at (1,0) (y2'') {$1$}; \node at (2,0) (y3'') {$1$}; \node at (3,0) (y4'') {$1$};
					\node at (4,0) (y5'') {$2$}; \node at (5,0) (y6'') {$1$}; \node at (6,0) (y7'') {$3$}; \node at (7,0) (y8'') {$3$};
					
					\node at (10,0) (z1'') {$1$}; \node at (11,0) (z2'') {$1$}; \node at (12,0) (z3'') {$1$}; \node at (13,0) (z4'') {$1$};
					\node at (14,0) (z5'') {$2$}; \node at (15,0) (z6'') {$1$}; \node at (16,0) (z7'') {$3$}; \node at (17,0) (z8'') {$3$};	 
					
						\draw  (a) -- (u1); \draw (a) -- (u2); \draw (3) -- (u3); \draw (3) -- (u4); \draw (1) -- (u5); \draw (1) -- (u6);
					
						\draw (u1) -- (x1); \draw (u1) -- (x2); \draw (u2) -- (x3); \draw (u2) -- (x4);
					\draw (u3) -- (x5); \draw (u3) -- (x6); \draw (u4) -- (x7); \draw (u4) -- (x8);
					
					\draw (u5) -- (x9); \draw (u5) -- (x10); \draw (u6) -- (x11); \draw(u6) -- (x12);
					
					\draw (x1) -- (y1); \draw (x1) -- (y2); \draw (x2) -- (y3); \draw (x2) -- (y4); \draw (x3) -- (y5); \draw (x3) -- (y6); \draw (x4) -- (y7); \draw (x4) -- (y8);
					
					\draw (x5) -- (z1); \draw (x5) -- (z2); \draw (x6) -- (z3); \draw (x6) -- (z4); \draw (x7) -- (z5); \draw (x7) -- (z6); \draw (x8) -- (z7); \draw (x8) -- (z8);
					
					\draw (x9) -- (w1); \draw (x9) -- (w2); \draw (x10) -- (w3); \draw (x10) -- (w4); \draw (x11) -- (w5); \draw (x11) -- (w6); \draw (x12) -- (w7); \draw (x12) -- (w8);
					
			\draw  (a') -- (u1'); \draw (a') -- (u2'); \draw (3') -- (u3'); \draw (3') -- (u4'); 
			
			\draw (u1') -- (x1'); \draw (u1') -- (x2'); \draw (u2') -- (x3'); \draw (u2') -- (x4');
			\draw (u3') -- (x5'); \draw (u3') -- (x6'); \draw (u4') -- (x7'); \draw (u4') -- (x8');

			\draw (x1') -- (y1'); \draw (x1') -- (y2'); \draw (x2') -- (y3'); \draw (x2') -- (y4'); \draw (x3') -- (y5'); \draw (x3') -- (y6'); \draw (x4') -- (y7'); \draw (x4') -- (y8');
			
			\draw (x5') -- (z1'); \draw (x5') -- (z2'); \draw (x6') -- (z3'); \draw (x6') -- (z4'); \draw (x7') -- (z5'); \draw (x7') -- (z6'); \draw (x8') -- (z7'); \draw (x8') -- (z8');

		
			\draw  (a'') -- (u1''); \draw (a'') -- (u2''); \draw (3'') -- (u3''); \draw (3'') -- (u4''); 
		
		\draw (u1'') -- (x1''); \draw (u1'') -- (x2''); \draw (u2'') -- (x3''); \draw (u2'') -- (x4'');
		\draw (u3'') -- (x5''); \draw (u3'') -- (x6''); \draw (u4'') -- (x7''); \draw (u4'') -- (x8'');

		\draw (x1'') -- (y1''); \draw (x1'') -- (y2''); \draw (x2'') -- (y3''); \draw (x2'') -- (y4''); \draw (x3'') -- (y5''); \draw (x3'') -- (y6''); \draw (x4'') -- (y7''); \draw (x4'') -- (y8'');
		
		\draw (x5'') -- (z1''); \draw (x5'') -- (z2''); \draw (x6'') -- (z3''); \draw (x6'') -- (z4''); \draw (x7'') -- (z5''); \draw (x7'') -- (z6''); \draw (x8'') -- (z7''); \draw (x8'') -- (z8'');
		
		\node at (2.8,24) (c1) {}; \node at (4.2,24) (c2) {};
		\draw [<->] (c1) -- (c2);
					\end{tikzpicture}
		
		\caption{An example of the sequence of moves.} \label{fig:switch}
		\end{figure}

			 			\begin{example}\label{ex:fig}
			 			{Figure \ref{fig:switch} shows a small example of this process that does two switches. 
			 				The transition matrix is $A=[110|001|100]$.
			 				In the first row of the figure we see three $\Delta_3$'s. 
			 				The first has its root at a site $x$ on $L_1$ in $\Delta_4$, and its labeling $\lambda^*$ is the restriction of a full labeling $\Lambda_4^1$ of $\Delta_4$; its root has label $a=1$.
			 				The second is the $\Delta_3$ that is rooted at site $x$ whose labeling $\lambda_1$ is the restriction of the labeling $\Lambda_4^2$ of $\Delta_4^2$ that is guaranteed to agree with $\Lambda_4^1$ on $L_4$; its root has label $b(a)=3$.
			 				The third is an arbitrary labeling of $\Delta_3$ that has $a=1$ at its root, as it should.}
			 			
			 			{We apply Switch and Replace as necessary on the first $\Delta_3$ to make it agree with the third, applying the same moves simultaneously on the second $\Delta_3$. 
			 				First we see that the symbols on the first row should be switched, so we switch the entire subtrees rooted at their sites. This produces the second row consisting of modifications of the first two $\Delta_3$'s. 
			 				Then we switch a single pair of sites on the last row. 
			 				At the end, the $\Delta_3$ on the left has the labeling $\lambda$ (illustrating Lemma \ref{lem:switch}), and the one on the right has $b(a)=3$ at the root and agrees with the one on the left (and hence the given $\lambda$) on $L_3$.}
			 		\end{example}

			 			We will use Proposition \ref{prop:n} to show that all relations $a \Rightarrow_n b(a)$ are found by the algorithm described above. 
			 			Recall that $\mathcal K_n$ denotes the set of relations that are known (have been discovered by the algorithm) up to time $n$, and $\mathcal D_n$ denotes the set of relations of degree $n$, the ones that can be ``realized on $\Delta_n$" (and not on $\Delta_m$ for any $m<n$).
			 			Clearly $\mathcal K_n \subseteq \cup_{m \leq n} \mathcal D_m$ for all $n \geq 0$. 
			 			We aim to prove that $\mathcal K_n = \cup_{m \leq n} \mathcal D_m$.
			 				
			 					Eventually the discovery process must end, but it is not clear that {\em all}~ the relations $i \Rightarrow j$ will have been discovered. 
			 					Some might get skipped when their turn to be discovered (presumably degree) arises.
			 			 			This might happen as follows. 
			 			 			Suppose that $i \Rightarrow_{n+1} j$. 
			 			Then given a labeling $\lambda$ of some $\Delta_{n+1}^i$ with an entry $a$ somewhere on the first line, there is a labeling $\lambda'$ of $\Delta_{n+1}^j$ which agrees with $\lambda$ on $L_{n+1}$. 
			 			We can read off the entry $b(a)$ of $\lambda'$ at the spot on $\Delta_{n+1}^j$ that corresponds to the spot labeled by $a$ in $\Delta_{n+1}^i$. 
			 			
			 			But why would we have $a \Rightarrow_n b(a)$? 
			 			For some labelings of $\Delta_n^a$ there might not be any labeling of $\Delta_n^{b(a)}$ such that the two agree on $L_n$ (although it works for the restrictions of $\lambda$ and $\lambda'$).
			 			Moreover, for fixed $a$ the choice of $b(a)$ might depend on the labeling $\lambda$, blocking the possibility that the relation $i \Rightarrow_{n+1} j$ arises from a relation such as $a \Rightarrow b(a)$. 
			 			The following theorem answers this question when the tree dimension $k$ is greater than or equal to the maximum row sum $s_A$ of the transition matrix $A$: 
			 			under that assumption, the algorithm described in Section \ref{sec:procedure} is guaranteed to find all relations among all possible initial symbols at the root.

			 		\begin{theorem}\label{thm:disc}
			 			{Suppose that $k \geq s_A$ and $i \Rightarrow_{n+1} j$ for some $n \geq 1$.
			 				Then the algorithm discovers this relation during or before round $n+1$.} 
			 			In other words, the height of any relation is less than or equal to its degree.
			 				\end{theorem}
			 		\begin{proof}
			 					{We use induction to show that when $k \geq s_A$ all existing relations are indeed found by the process.
			 					For each $n \geq 0$ denote by $\mathcal D_n$ the set of relations of degree $n$, and by $\mathcal K_n$ the set of relations of height no more than $n$.
			 					Let $\mathcal D=\cup_n \mathcal D_n$ and $\mathcal K=\cup_n \mathcal K_n$.
			 					For $n=0$ all existing relations $i \Rightarrow_0 i$ are found. 
			 					For $n=1$, all existing relations $i \Rightarrow_1 j$ are found by examining whether or not $A_i \leq A_j$.}
			 				Suppose now that $n \geq 1$ and {$\mathcal D_m \subseteq \mathcal K_m$ for $0 \leq m \leq n$.} 
			 				Given a relation $(i \Rightarrow_{n+1} j) \in \mathcal D_{n+1}$, apply Proposition \ref{prop:n} (1) to find for every $a \in S(i)$ a symbol $b=b(a) \in S(j)$ such that $a \Rightarrow_n b$. 
			 				These relations are {\em a priori} in {$\cup_{m\leq n}\mathcal D_m$}, but by the induction hypothesis we have {$\cup_{m \leq n}\mathcal D_m \subseteq \mathcal K_n$.} 
			 				Thus when we apply Proposition \ref{prop:n} (2) to conclude, as part of our process
			 				{(on Round $n+1$)}, that $i \Rightarrow_{n+1} j$, that relation is {added to $\mathcal K_{n+1}$.}
			 				It follows that 
			 				{$\mathcal D_{n+1} \subseteq \mathcal K_{n+1}$, and hence}
			 				$\mathcal D = \mathcal K$ when $k \geq s_A$.
			 							 		\end{proof}

			 			\begin{example}\label{ex:smallk}
			 				The hypothesis in Theorem \ref{thm:disc} that $k \geq s_A$ is essential. 
			 					{If $k < s_A$, it might happen in round $n+1$ that 
			 					for some $a \in S(i)$ there is no $b \in S(j)$ such that $a \Rightarrow_n b$, yet
			 					$i \Rightarrow_{n+1} j$.
			 					This occurs for $k=2$ and $s_A=3$ in the example
			 					\be
			 					A=[0111|1011|1101|1110],
			 					\en
			 					for which there are no relations of degree $1$, yet 
			 					$i \Rightarrow_2 j$ for all $i,j \in \{1,2,3,4\}$.}
			 					(When $k \geq 3$, $A \notin \cup_n P(k,n)$.)
			 			\end{example}

			 			 \section{further comments}\label{sec:comments}
	
		{\bf Computing time.}	In the case that $k \geq s_A$, because of Theorem \ref{thm:disc}, each relation $i \Rightarrow _{n+1} j$ is determined by the relations $a \Rightarrow_n b$ already in hand according to the process described above.
			Therefore, given $i,j$, the question of whether $i \Rightarrow_n  j$ for some $n$ is eventually answered, at least when $k \geq s_A$. 			
			But how to bound the number of rounds required?
		Given that $i \Rightarrow_{n+1} j$, it might take through round $n$ to discover that relation. 
		
		{However, given $A$ there are at most $d^2$ relations $i \Rightarrow j$, each with a finite degree, so their degrees are bounded. 
			Since the height of each relation equals its degree, all relations are eventually discovered in a finite time determined by $A$.}
		
		The relations matrix $R$ has $d^2$ entries, and on each round at least one entry is changed from $0$ to $1$. 
		Therefore the algorithm terminates after no more than $d^2$ rounds, answering decisively (when $k \geq s_A$) whether or not $A \in \cup_n P(k,n)$.  
		It is also possible to estimate the maximum possible number of computations that might be required in each round.

	{\bf Matrix products.}	One may view the problem of determining what configurations can appear on a target set given a starting symbol in terms of multiplication of transition matrices. 
	{Suppose for now that $k=2$.}
		For each $n=0,1, \dots$ denote by $\Lambda_n$ the set of all labelings of the $n$'th row $L_n$ of $\Delta_n$ by elements of the alphabet $D=\{1,\dots ,d\}$. 
		Thus $\Lambda_0=D, \Lambda_1=D^2, \dots , \Lambda_n=D^{2^n}$. 
		Fixing a $d \times d$ $0,1$ matrix $A$, denote by $A_n, n \geq 0$, the $|\Lambda_n| \times |\Lambda_{n+1}|$ matrix specifying, by means of its $0,1$ entries, the allowed transitions from $\Lambda_n$ to $\Lambda_{n+1}$.
		
		For example, if 
		\be
	A=	\begin{bmatrix}
			1 &1\\
			1 &0
			\end{bmatrix}
			=A_0, \quad\text{ then } \quad
A_1= \begin{matrix}
	&11 &12 &21 &22\\
1   &1  &1  &1  &1 \\
2   &1  &0  &0  &0
\end{matrix}
\en	
 $A_2$ is $4 \times 16$, etc. 

	The entries in the product $A_0 \cdots A_n$ will tell the numbers of paths from each symbol in $D$ (indexing the rows) to each labeling of $L_n$ (indexing the columns). 	
	If the rows are identical mod $0$, meaning that they all have positive entries in exactly the same places, then $A \in P(2,n)$. 
	If every row is positive, then $A \in P^*(2,n)$.

	So the problem of deciding which $A \in P(2,n)$ for some $n$ is solved in principle.
	But how can we answer the question for a given $A$ without multiplying out $A_0 \cdots A_n$ and examining the result? 
	Since the matrices grow exponentially, maybe we are up against a growth problem as in vector addition or quantum reachability? 
	{It's even worse when $k>2$.}
	{The algorithm presented in Section \ref{sec:procedure} answers this question when $k \geq s_A$.}

	{\bf Questions.} 1. How to determine whether a given matrix is in $P(k)=\cup_n P(k,n)$ or not in the case when the tree dimension is less than the maximum row sum is an obvious question for further investigation.
	
	2, As mentioned above, the situation studied here, namely a tree and target set the $n$'th row $L_n$, motivated by computation of pressure on trees, is a simple first case. 
	One may consider the Cayley graph of any semigroup, not necessarily free, or a more general tree, or even an arbitrary graph, and an arbitrary target set, and seek to understand properties such as fairness or completeness of a matrix that specifies allowed transitions between labels of adjacent sites.

	{\bf Declarations:} The authors have no relevant financial, non-financial, or competing interests to declare.

 \end{document}